\documentclass[10pt]{article}
\usepackage[utf8]{inputenc}

\usepackage{amssymb,amsthm,amsmath}
\usepackage{enumerate}
\usepackage{graphicx,color}
\usepackage[hidelinks]{hyperref}

\newcommand{\dd}{\mathrm{d}}
\newcommand{\E}{\mathbb{E}}

\newcommand{\R}{\mathbb{R}}

\newcommand{\p}[1]{\mathbb{P}\left( #1 \right)}

\DeclareMathOperator{\sgn}{sgn}

\usepackage[paper=a4paper, left=1.3in, right=1.3in, top=1in, bottom=1in]{geometry}
\newtheorem{theorem}{Theorem}
\newtheorem{lemma}[theorem]{Lemma}
\newtheorem{corollary}[theorem]{Corollary}

\theoremstyle{remark}
\newtheorem{remark}[theorem]{Remark}

\theoremstyle{definition}

\title{\vspace{-3em}Sharp Khinchin-type inequalities for symmetric discrete uniform random variables}
\author{Alex Havrilla\thanks{Carnegie Mellon University; Pittsburgh, PA 15213, USA. Email: alumhavr@andrew.cmu.edu}
\ and 
Tomasz Tkocz\thanks{Carnegie Mellon University; Pittsburgh, PA 15213, USA. Email: ttkocz@math.cmu.edu. Research supported in part by the Collaboration Grants from the Simons Foundation and NSF grant DMS-1955175.}
}
\date{15th October 2020}

\begin{document}

\maketitle

\begin{abstract}
We establish several optimal moment comparison inequalities (Khinchin-type inequalities) for weighted sums of independent identically distributed symmetric discrete random variables which are uniform on sets of consecutive integers. Specifically, we obtain sharp constants for the second moment and any moment of order at least $3$ (using convex dominance by Gaussian random variables). In the case of only $3$ atoms, we also establish a Schur-convexity result. For moments of order less than $2$, we get sharp constants in two cases by exploiting Haagerup's arguments for random signs.
\end{abstract}

\bigskip

\begin{footnotesize}

\end{footnotesize}

\bigskip

\section{Introduction}

The classical Khinchin inequality asserts that all moments of weighted sums of independent random signs are comparable (see \cite{Khi}). More specifically, if we consider independent random signs $\varepsilon_1, \varepsilon_2, \ldots$, the probability of each $\varepsilon_i$ taking the value $\pm 1$ is a half and form a weighted sum $S = \sum_{i=1}^n a_i\varepsilon_i$ with real coefficients $a_i$, then for every $p, q > 0$, there is a positive constant $C_{p,q}$ independent of $n$ and the $a_i$ such that
\begin{equation}\label{eq:CK}
\|S\|_p \leq C_{p,q}\|S\|_q.
\end{equation}
As usual, $\|X\|_p = (\E|X|^p)^{1/p}$ denotes the $p$-th moment of a random variable $X$. Moment comparison inequalities like this one are well understood up to universal constants in a great generality due to Lata\l a's formula from \cite{Lat-mom}. They have found numerous applications in classical results in analysis (for example in the proof of the Littlewood-Payley decomposition or Grothendieck's inequality) and, especially their extensions to vector valued settings (Kahane's inequalities), have been widely used in (local) theory of Banach spaces (see \cite{LT}, \cite{MSch}). One of the major challenges is to find the best constants $C_{p,q}$, which has attracted considerable attention and has important applications (for instance in geometry, $C_{2,1}$ is directly linked with the maximum volume projections of the $n$\nobreakdash-dimensional cross-polytope onto $n-1$ dimensional subspaces, see \cite{Ball, BN}). 
Besides, attacking sharp inequalities forces us to uncover often deep and effective mechanisms explaning  \emph{bigger pictures} and providing insights as to why certain inequalities are true.

For results concerning the best constant $C_{p,q}$ in the classical Khinchin inequality \eqref{eq:CK}, we mention in passing works \cite{Eat, Haa, Kom, LO-best, Mo, NO, NP, koles-b, Pin, Sz, Whi, Y}, highlighting only that the optimal value of $C_{p,q}$ is known when $p < q$ (trivial), either $p$ or $q$ is $2$, or both $p$ and $q$ are even.  
There have been only a handful of results concering random variables other than random signs. They involve continuous random variables uniformly distributed on symmetric intervals and generalisations for random vectors uniformly distributed on Euclidean spheres and balls (see \cite{BC, Kon, KK, LO}), mixtures of centred Gaussians (see \cite{AH, ENT1}), the so-called \emph{exponential family} given by the density $e^{-|x|^\alpha}$ and uniform distributions on unit $\ell_\alpha^n$ balls (see \cite{BGMN, ENT1, ENT2}), dependent random signs (see \cite{PS, Spe}), as well as general random variables via their \emph{spectral properties}  (see \cite{KLO,koles}). 
 
This paper concerns Khinchin-type inequalitites with sharp constant for symmetric discrete random variables, generalising random signs by allowing more than just two atoms. Specifically, in the simplest case, let $L$ be a positive integer and let $X$ be uniform on the set $\{-L,\ldots,-1\}\cup\{1,\ldots,L\}$. What are best constants in moment comparison inequalities for weighted sums of independent copies of $X$? Note that the following two  \emph{extreme} cases have been understood: when $L=1$, $X$ is a symmetric random sign discussed above, whereas when $L \to \infty$, $X/L$ converges in distribution to a random variable uniform on $[-1,1]$, the case analysed in \cite{LO}.

We present our results in the next section and then proceed with their proofs in their order of statement. We say that a random variable $X$ is \emph{symmetric} if $-X$ has the same distribution as $X$, equivalently $\varepsilon X$ and $\varepsilon|X|$ have the same distribution as $X$, where $\varepsilon$ is an independent \emph{symmetric random sign}, that is $\p{\varepsilon = -1} = \p{\varepsilon = 1} = \frac{1}{2}$. We usually denote by $G$ a standard Gaussian random variable, that is a real-valued random variable with density $\frac{1}{\sqrt{2\pi}}e^{-x^2/2}$. For $p > 0$, we have $\|G\|_p = 2^{1/2}\left(\pi^{-1/2}\Gamma\left((p+1)/2\right)\right)^{1/p}$, where $\Gamma$ stands for the gamma function. 

\paragraph{Acknowledgements.} 
We are indebted to Krzysztof Oleszkiewicz for his help and valuable feedback.

\section{Results}

Given $\rho_0 \in [0,1]$ and a positive integer $L$, consider a random variable $X$ with
\begin{equation}\label{eq:def-X}
\p{X = 0} = \rho_0 \quad \text{and} \quad \p{X = j} = \p{X = -j} = \frac{1-\rho_0}{2L} \quad \text{for } j = 1, \ldots, L.
\end{equation}
For $a = (a_1, \ldots, a_n) \in \R^n$ and $p \geq 1$, we let
\begin{equation}\label{eq:def-N}
N_p(a) = \left\|\sum_{i=1}^n a_iX_i\right\|_p,
\end{equation}
where $X_1, X_2, \ldots$ are i.i.d. copies of $X$. Throughout, $G$ stands for a standard Gaussian random variable. We refer to the classical monograph \cite{HLP}, or to \cite{Bh} for a concise exposition of majorisation and Schur-convexity. Our main results are as follows.

\begin{theorem}\label{thm:Schur}
Let $\rho_0 \in [0,\frac{1}{2}]$ and $L = 1$. If $p \geq 3$, then the function
$
(a_1,\ldots,a_n) \mapsto N_p(\sqrt{a_1},\ldots,\sqrt{a_n})
$
is Schur-concave on $[0,+\infty)^n$.
\end{theorem}

As an immediate corollary, we obtain best constants in Khinchin inequalities (it can be done as, for instance, in the proof of Corollary 25 from \cite{ENT1}).

\begin{corollary}
Under the assumptions of Theorem \ref{thm:Schur}, the best constant $C_p$ such that the inequality $N_{p}(a) \leq C_pN_{2}(a)$ holds for all $n$ and $a \in \R^n$ is $C_p =\|G\|_p$.
\end{corollary}

Our next result concerns Khinchin inequalities for $p \geq 3$ for arbitrary $L$ and $\rho_0 = 0$. 

\begin{theorem}\label{thm:2-p>3}
Let $\rho_0 = 0$ and $L \geq 1$ be an integer. If $p \geq 3$, the the best constant $C_p$ such that the inequality $N_{p}(a) \leq C_pN_{2}(a)$ holds for all $n$ and $a \in \R^n$ is $C_p = \|G\|_p$.
\end{theorem}

Finally, in the presence of large mass at $0$ and arbitrarily many atoms $L$, we obtain a sharp $L_1-L_2$ inequality, which holds in a greater generality.

\begin{theorem}\label{thm:L1-L2}
Let $\rho_0 \in [\frac{1}{2},1)$ and let $Y, Y_1, Y_2, \ldots$ be i.i.d. symmetric random variables with $\p{Y = 0} = \rho_0$. Define $N_p(a) = \|\sum_{i=1}^n a_iY_i\|_p$. The best constant $c_1$ such that the inequality $N_{1}(a) \geq c_1N_{2}(a)$ holds for all $n$ and $a \in \R^n$ is $c_1 = \|Y\|_1/\|Y\|_2$.
\end{theorem}

Some restrictions on $\rho_0$ in our theorems are needed, however our specific ones may not be optimal. We defer a discussion to the last section.

\begin{remark}
When $p$ is a positive even integer, Theorem \ref{thm:2-p>3} can be deduced from Newman's results from \cite{New1} (see also \cite{New2}).
\end{remark}

\begin{remark}\label{rem:even}
Using \cite{NO}, Theorem \ref{thm:2-p>3} can be extended to a sharp moment comparison between all even moments with an \emph{optimal} restriction on $\rho_0 $, which will appear elsewhere.
\end{remark}

\begin{remark}
When $L=1$ and $Y=X$, Theorem \ref{thm:L1-L2} follows from general results of Oleszkiewicz from \cite{koles} concerning arbitrary symmetric random variables and coefficients in Banach space (see Corollary 2.4 therein). 
\end{remark}

We finish this section with a few words on proofs. Our proof of Theorem \ref{thm:Schur} follows a direct approach from Eaton's work \cite{Eat}, combined with techniques (used for instance in \cite{FHJSZ}, or  \cite{ENT2}) exploiting linearity and allowing to reduce verification of certain inequalities needed for averages of power functions $|\cdot|^p$ to \emph{simple} (piecewise linear) functions. To prove Theorem \ref{thm:2-p>3}, we employ an inductive argument (on $n$) which crucially uses independence and convexity of certain functions and is based on swapping the $X_i$ one by one with independent Gaussians.
For Theorem \ref{thm:L1-L2}, we extend Haagerup's short proof from \cite{Haa} of Szarek's result from \cite{Sz} saying that the best constant $C_{2,1}$ in \eqref{eq:CK} is $\sqrt{2}$ (for the latter, see also \cite{LO-best, Lit, Tom}).
We rely on an integral representation for the first moment used by Haagerup, combined with convexity arguments allowing to handle more atoms.


\section{Proofs}\label{sec:proofs}

\subsection{A Schur-convexity result: Proof of Theorem \ref{thm:Schur}}

We begin with two technical lemmas. Let $\mathcal{C}$ be the linear space of all continuous functions on $\R$ equipped with pointwise topology. Let $\mathcal{C}_{1} \subset \mathcal{C}$ be the cone of all odd functions on $\R$ which are nondecreasing convex on $(0,+\infty)$ and let $\mathcal{C}_{2} \subset \mathcal{C}$ be the cone of all even functions on $\R$ which are nondecreasing convex on $(0,+\infty)$. Note that $\mathcal{C}_{2}$ is the closure (in the pointwise topology) of the set $\mathcal{S} = \{(|x|-\gamma)_+, \ \gamma \geq 0\}$ .

\begin{lemma}\label{lm:r-is-convex}
Let $q \geq 2$, $w \geq 0$ and $\phi_w(x) = \sgn(x+w)|x+w|^q + \sgn(x-w)|x-w|^q$, $x \in \R$. Then $\phi_w \in \mathcal{C}_1$. Let $r_w(x) = \frac{\phi_w(x)}{x}$, $x \in \R$ (with the value at $x=0$ understood as the limit). Then $r_w \in \mathcal{C}_2$.
\end{lemma}
\begin{proof}
The case $w=0$ is clear. For $w > 0$, verifying that $\phi_w \in \mathcal{C}_1$ and $r_w \in \mathcal{C}_2$, by homogeneity, is equivalent to doing so for $w=1$. Let $w=1$ and denote $\phi=\phi_1$ and $r=r_1$. Suppose we have shown that $r \in \mathcal{C}_2$. Then, plainly, $\phi(x) = xr(x)$ is also nondecreasing on $(0,\infty)$ and $\phi''(x) = (r(x) +xr'(x))' = 2r'(x) + xr''(x)$ is nonnegative on $(0,\infty)$ since $r'$ and $r''$ are nonnegative on $(0,\infty)$. 

It remains to prove that $r \in \mathcal{C}_2$. Plainly $\phi(x)$ is odd and thus $r(x)$ is even. Thus we consider $x > 0$.

\bigskip
\noindent
\emph{Case 1.} $x \geq 1$. We have, $\phi(x) = (x+1)^q + (x-1)^q$,
\[
r'(x) = \frac{\phi'(x)}{x} - \frac{\phi(x)}{x^2} = q\frac{(x+1)^{q-1}+(x-1)^{q-1}}{x} - \frac{(x+1)^q+(x-1)^q}{x^2}
\]
and
\begin{align*}
x^3r''(x) &= x^3\Bigg[\frac{\phi''(x)}{x}-2\frac{\phi'(x)}{x^2}+2\frac{\phi(x)}{x^3}\Bigg]=q(q-1)x^2\Big[(x+1)^{q-2}+(x-1)^{q-2}\Big] \\
&\qquad\qquad\qquad\qquad\qquad-2qx\Big[(x+1)^{q-1}+(x-1)^{q-1}\Big]+ 2\Big[(x+1)^q+(x-1)^q\Big].
\end{align*}
Taking one more derivative gives
\[
(x^3r''(x))' = q(q-1)(q-2)x^2\Big[(x+1)^{q-3}+(x-1)^{q-3}\Big]
\]
which is clearly positive for $x > 1$ since $q \geq 2$. Thus, for $x > 1$, we have
\[
x^3r''(x) > r''(1) = q(q-1)\cdot 2^{q-2}-2q\cdot 2^{q-1}+2\cdot 2^{q} = 2^{q-2}\left( \left(q-\frac{5}{2}\right)^2 + \frac{7}{4}\right) > 0.
\]
Therefore, $r''(x) > 0$ for $x > 1$. Since $r'(1) = q2^{q-1}-2^q = 2^{q-1}(q-2) \geq 0$, we also get that $r'(x)$ is positive for $x > 1$.

\bigskip
\noindent
\emph{Case 2.} $0 < x < 1$. The argument and the computations are very similar to Case 1. We have, $\phi(x) = (1+x)^q - (1-x)^q$, we find that 
$(x^3r''(x))' = q(q-1)(q-2)x^2\Big[(1+x)^{q-3}+(1-x)^{q-3}\Big].$
If $q > 2$, this is positive for $0 < x  < 1$. Then in this case, consequently, $x^3r''(x) > x^3r''(x)\Big|_{x=0} = 0$, so $r''(x)$ is positive for $0 < x < 1$. As a result, $r'(x) > r'(0+) = 0$ for $0 < x < 1$. If $q = 2$, we simply have $\phi(x) = 4x$ and $r(x) = 4$. 

Combining the cases, we see that both $r'$ and $r''$ are nonnegative on $(0,+\infty)$, which finishes the proof.
\end{proof}


\begin{lemma}\label{lm:2point-C}
The best constant $D$ such that the inequality
\begin{equation}\label{eq:2point-C}
D\cdot\left[\frac{\phi(a+b)-\phi(b-a)}{2a} - \frac{\phi(a+b)+\phi(b-a)}{2b}\right] \geq \left[ \frac{\phi(b)}{b}-\frac{\phi(a)}{a}\right]
\end{equation}
holds for all $0 < a < b$ and every function $\phi(x)$ of the form $xr(x)$, $r \in \mathcal{C}_2$, is $D=1$.
\end{lemma}
\begin{proof}
For $\phi(x) = xr(x)$, $r(x) = |x|$, by homogeneity, inequality \eqref{eq:2point-C} is equivalent to: for all $0 < a < 1$, we have
\[
D\cdot\left[\frac{(1+a)^2-(1-a)^2}{2a} - \frac{(1+a)^2+(1-a)^2}{2}\right] \geq 1-a,
\]
that is $D\cdot(1-a^2) \geq (1-a)$ for all $0 < a < 1$, which holds if and only if $D \geq 1$. Now we show that in fact \eqref{eq:2point-C} holds with $D=1$ for every $\phi(x) = xr(x)$, where $r \in \mathcal{C}_2$. Since $\mathcal{C}_2$ is the closure of $\mathcal{S}$, by linearity, it suffices to show this for all simple functions $r \in \mathcal{S}$, that is $r(x) = (|x|-\gamma)_+$. By homogeneity, this is equivalent to showing that for all $\gamma \geq 0$ and $0 < a < 1$, we have
\begin{align*}
&\frac{(1+a)(1+a-\gamma)_+-(1-a)(1-a-\gamma)_+}{2a} - \frac{(1+a)(1+a-\gamma)_++(1-a)(1-a-\gamma)_+}{2} \\
&\qquad\geq  (1-\gamma)_+-(a-\gamma)_+.
\end{align*}
Fix $0 < a < 1$. Let $h_a(\gamma)$ be the left hand side minus the right hand side. For $\gamma \geq 1+a$, $h_a(\gamma) = 0$. Since as a function of $\gamma$, $h_a(\gamma)$ is piecewise linear, showing that it is nonnegative on $[0,1+a]$ is equivalent to verifying it at the nodes $\gamma \in \{0, 1, a, 1-a\}$. We have, $h_a(0) = a-a^2 > 0$. Next, $h_a(1) = \frac{(1+a)a}{2a} - \frac{(1+a)a}{2} = \frac{1}{2}(1+a)(1-a) > 0$. Finally, to check $\gamma = a$ and $\gamma = 1-a$, we consider two cases.

\bigskip
\noindent
\emph{Case 1.} $a \leq 1-a$, that is $0< a \leq \frac{1}{2}$. Then,
\[
h_a(a) = \frac{(1+a)-(1-a)(1-2a)}{2a} - \frac{(1+a)+(1-a)(1-2a)}{2} - (1-a) = a(1-a) > 0
\]
and
\[
h_a(1-a) = \frac{(1+a)2a}{2a} - \frac{(1+a)2a}{2}-a = 1-a^2-a \geq 1 - \frac{1}{4} - \frac{1}{2} = \frac{1}{4}.
\]

\bigskip
\noindent
\emph{Case 2.} $a > 1-a$, that is $\frac{1}{2} < a <1$. Then,
\[
h_a(a) = \frac{(1+a)}{2a} - \frac{(1+a)}{2} - (1-a) = \frac{(1-a)^2}{2a} > 0
\]
and
\[
h_a(1-a) = \frac{(1+a)2a}{2a} - \frac{(1+a)2a}{2}-[a-(2a-1)] = a(1-a) > 0.
\]
\end{proof}

\begin{proof}[Proof of Theorem \ref{thm:Schur}]
Fix $p \geq 3$ and let $F(x) = |x|^p$. We would like to show that the function
\[
\Phi(a_1,\ldots,a_n) = \E F\left(\sum_{i=1}^n \sqrt{a_i}X_i\right)
\]
is Schur concave. Since $\Phi$ is symmetric, by Ostrowski's criterion (see, e.g., Theorem II.3.14 in \cite{Bh}), $\Phi$ is Schur concave if and only if
$
\frac{\partial \Phi}{\partial a_1} \geq \frac{\partial \Phi}{\partial a_2}$, $a_1 < a_2$,
which is equivalent to
\[
\frac{1}{\sqrt{a_1}}\E[ X_1F'(S)] \geq \frac{1}{\sqrt{a_2}}\E[X_2 F'(S)],
\]
where $S = \sqrt{a_1}X_1+\sqrt{a_2}X_2 + W$ and $W = \sum_{i>2} \sqrt{a_i}X_i$.
We take the expectation with respect to $X_1$ and $X_2$.
Suppose $\rho_0 < 1$.
Since $F'$ is odd and $W$ is symmetric, we get, $-\E F'(-\sqrt{a_1}+W) = \E F'(\sqrt{a_1}+W)$ and similarly for the other terms that show up. Consequently, the inequality can be equivalently rewritten as
\begin{align*}
&\quad \frac{1}{\sqrt{a_1}}\Bigg(2\rho_0\E F'(\sqrt{a_1}+W) + (1-\rho_0) \E[ F'(\sqrt{a_1}+ \sqrt{a_2} +  W) - F'(-\sqrt{a_1} + \sqrt{a_2} + W) ]\Bigg) \\
&\geq \frac{1}{\sqrt{a_2}}\Bigg(2\rho_0\E F'(\sqrt{a_2}+W) +(1-\rho_0) \E[ F'(\sqrt{a_2}+ \sqrt{a_1} +  W) + F'(\sqrt{a_2} - \sqrt{a_1} + W) ]\Bigg).
\end{align*}
Set $a = \sqrt{a_1}$, $b = \sqrt{a_2}$ and
\[
\phi(x) = \E F'(x+W), \qquad x \in \R
\]
($\phi$ is also odd). Suppose $\rho_0 > 0$. Then, the validity of the above inequality is equivalent to the question whether for all $0 < a < b$,
\begin{equation}\label{eq:phi-2point}
(\rho_0^{-1}-1)\left[\frac{\phi(a+b)-\phi(b-a)}{2a} - \frac{\phi(a+b)+\phi(b-a)}{2b}\right] \geq \left[ \frac{\phi(b)}{b}-\frac{\phi(a)}{a}\right].
\end{equation}
By the symmetry of $W$, it has the same distribution as $\varepsilon |W|$, where $\varepsilon$ is an independent symmetric random sign, so we can write $\phi(x) = \frac{1}{2}\E\phi_{|W|}(x)$, where for $w \geq 0$, we set $\phi_w(x) = F'(x+w) + F'(x-w)$. By Lemmas \ref{lm:r-is-convex} and \ref{lm:2point-C}, inequality \eqref{eq:phi-2point} holds for $\phi_w$ in place of $\phi$ (for every $w \geq 0$) as long as $\rho_0^{-1} - 1 \geq 1$. Taking the expectation against $|W|$ yields the inequality for $\phi$, as desired. For $\rho_0 =0$, we can for instance argue by taking the limit $\rho_0 \to 0+$.
\end{proof}

\subsection{Second, third and higher moments: Proof of Theorem \ref{thm:2-p>3}}

The value of $C_p$ is seen to be sharp by taking $a_1 = \ldots = a_n = \frac{1}{\sqrt{n}}$, letting $n \to \infty$ and applying the central limit theorem.


To establish $N_{p}(a) \leq C_pN_2(a)$, we set
\begin{equation}\label{eq:def-sigma}
\sigma  = \sqrt{\E |X_1|^2} = \left(\frac{(L+1)(2L+1)}{6}\right)^{1/2}
\end{equation}
and let $G_1, G_2, \ldots$ be i.i.d. centred Gaussian random variables with variance $\sigma^2$. Since
\[
C_p^p\left(\E\left|\sum_{i=1}^n a_iX_i \right|^2\right)^{p/2} = C_p^p\left(\sum_{i=1}^n a_i^2\right)^{p/2}\sigma^{p/2} = \E\left|\sum_{i=1}^n a_iG_i \right|^p,
\]
inequality $N_p(a) \leq C_pN_2(a)$ is equivalent to
\[
 \E\left|\sum_{i=1}^n a_iX_i \right|^p \leq  \E\left|\sum_{i=1}^n a_iG_i \right|^p.
\]
By independence and induction, it suffices to show that for every reals $a, b$, we have
\begin{equation}\label{eq:XvsG}
\E|a+bX_1|^p \leq \E|a+bG_1|^p.
\end{equation}
This will follow from the following claim.

\bigskip
\noindent\textbf{Claim.} For every convex nondecreasing function $h\colon [0,+\infty)\to [0,+\infty)$, we have 
\begin{equation}\label{eq:X^2vsG^2}
\E h(X_1^2) \leq \E h(G_1^2).
\end{equation}

\noindent
Indeed, \eqref{eq:XvsG} for $b = 0$ is clear. Assuming $b \neq 0$, by homogeneity, \eqref{eq:XvsG} is equivalent to
\[
\E|a+X_1|^p \leq \E|a+G_1|^p.
\]
Using the symmetry of $X_1$, we can write
\[
2\E|a+X_1|^p = \E|a + |X_1||^p + \E|a-|X_1||^p = \E h_a(X_1^2),
\]
where
\begin{equation}\label{eq:def-h_a}
h_a(x) = |a + \sqrt{x}|^p + |a - \sqrt{x}|^p, \qquad x \geq 0
\end{equation}
(and similarly for $G_1$). The convexity of $h_a$ is established in the following standard lemma (see also e.g. Proposition 3.1 in \cite{FHJSZ}).

\begin{lemma}\label{lm:h_a-convex}
Let $p \geq 3$, $a \in \R$. Then $h_a$ defined in \eqref{eq:def-h_a} is convex nondecreasing on $[0,\infty)$.
\end{lemma}
\begin{proof}
The case $a = 0$ is clear (and the assertion holds for $p \geq 2$). The case $a \neq 0$ reduces by homogeneity to, say $a = 1$. We have
\[
h_1'(x) = \frac{p}{2\sqrt{x}}\Big[|1+\sqrt{x}|^{p-1}+\text{sgn}(\sqrt{x}-1)|\sqrt{x}-1|^{p-1}\Big]
\]
and it suffices to show that the function $g(y) = \frac{|1+y|^{p-1}+\text{sgn}(y-1)|y-1|^{p-1}}{y}$ is nondecreasing on $(0,\infty)$. Call the numerator $f(y)$. Since $g(y) = \frac{f(y) - f(0)}{y-0}$, it suffices to show that $f$ is convex $(0,\infty)$. We have $f'(y) = (p-1)(|1+y|^{p-2}+|y-1|^{p-2})$ which is convex on $\R$ for $p \geq 3$, hence nondecreasing on $(0,\infty)$ (as being even). This justifies that $h_1'$ is nondecreasing, hence $h_1$ is convex. Since $h_1'(0) = f'(0) = 2(p-1) > 0$, we get $h_1'(x) \geq h_1'(0) > 0$, so $h_1$ is increasing on $(0,\infty)$.
\end{proof}

Thus $2\E|a+X_1|^p = \E h_a(X_1^2) \leq \E h_a(G_1^2) = 2\E|a+G_1|^p$ by the claim, as desired. It remains to prove the claim.

\begin{proof}[Proof of the claim.]
When $L=1$, the claim follows immediately because $X_1^2 = 1$ and by Jensen's inequality, $\E h(G_1^2) \geq h(\E G_1^2) = h(1) = \E h(X_1^2)$. We shall assume from now on that $L \geq 2$.

By standard approximation arguments, it suffices to show that the claim holds for $h(x) = (x-a)_+$ for every $a > 0$. Here and throughout $x_+ = \max\{x,0\}$. Note that 
\[
\mathbb{E}(X_1^2-a)_+ = \frac{1}{2L}\sum_{k=-L}^L(k^2-a)_+ = \frac{1}{L}\sum_{k = \lceil \sqrt{a} \rceil}^L (k^2-a)
\]
and
\[
\mathbb{E}(G_1^2-a)_+ = \int_{-\infty}^{\infty}(x^2-a)_+\frac{1}{\sqrt{2 \pi \sigma^2}}e^{-x^2/2\sigma^2}\dd x =\sqrt{\frac{2}{\pi \sigma^2}}\int_{\sqrt{a}}^{\infty}(x^2-a)e^{-x^2/2\sigma^2}\dd x
\]
with $\sigma$ (depending on $L$) defined by \eqref{eq:def-sigma}.
Fix an integer $L \geq 2$ and set for nonnegative $a$,
\[
f(a) = \sqrt{\frac{2}{\pi \sigma^2}}\int_{\sqrt{a}}^{\infty}(x^2-a)e^{-x^2/2\sigma^2}\dd x-\frac{1}{L}\sum_{k = \lceil \sqrt{a} \rceil}^L (k^2-a).
\]
Our goal is to show that $f(a) \geq 0$ for every $a \geq 0$. This is clear for $a > L^2$ because then the second term is $0$. Note that $f$ is continuous (because $x \mapsto x_+$ is continuous). For $a \in (b^2, (b+1)^2)$ with $b \in \{0,1,\ldots,L-1\}$ our expression becomes 
\[
f(a) = \sqrt{\frac{2}{\pi \sigma^2}}\int_{\sqrt{a}}^{\infty}(x^2-a)e^{-x^2/2\sigma^2}dx-\frac{1}{L}\sum_{k=b+1}^L (k^2-a),
\]
is differentiable and
\begin{align}\label{eq:f'}
f'(a) &=  -\sqrt{\frac{2}{\pi \sigma^2}}\int_{\sqrt{a}}^{\infty}e^{-x^2/2 \sigma^2}\dd x- \frac{1}{L}\sum_{k = b+1}^L (-1) = -\sqrt{\frac{2}{\pi \sigma^2}}\int_{\sqrt{a}}^{\infty}e^{-x^2/2 \sigma^2}\dd x + \frac{L-b}{L}.
\end{align}
Bounding $b < \sqrt{a}$ yields
\begin{align*}
f'(a) &\geq -\sqrt{\frac{2}{\pi \sigma^2}}\int_{\sqrt{a}}^{\infty}e^{-x^2/2 \sigma^2}\dd x + \frac{L-\sqrt{a}}{L} = -\sqrt{\frac{2}{\pi}}\int_{\sqrt{a}/\sigma}^{\infty}e^{-x^2/2}\dd x + \left(1 -  \frac{\sqrt{a}}{L}\right).
\end{align*}
Let $\tilde g(a)$ denote the right hand side.
We have obtained $f' \geq \tilde g$ on $(0,L^2)$ (except for the points $1^2, 2^2, \ldots$). Since $f$ is absolutely continuous and $f(0) = 0$, we can write $f(a) = \int_0^a f'(x) \dd x$ and consequently
\[
f(a) \geq g(a), \qquad a \in [0,L^2],
\]
where we define
\[
g(a) = \int_0^a \tilde g(x)\dd x.
\]
Note: $g''(a) = \tilde g'(a) = \frac{1}{2\sqrt{a}}\left(\sqrt{\frac{2}{\pi}}\frac{1}{\sigma}e^{-\frac{a}{2\sigma}} - \frac{1}{L}\right)$ which changes sign from positive to negative (since $\sqrt{\frac{2}{\pi}}\frac{1}{\sigma} - \frac{1}{L} > 0$ for $L \geq 2$). This implies that $g'$ is first strictly increasing, then strictly decreasing and together with $g'(0) = \tilde g(0) = 0$, $g'(\infty) = -\infty$, it gives that $g'$ is first positive, then negative. Consequently, $g$ is first strictly increasing and then strictly decreasing. Since $g(0) = 0$, to conclude that $g$ is nonnegative on $[0,L^2]$ (hence $f$), it suffices to check that $g(L^2) \geq 0$. We have,
\begin{align*}
g(L^2) &= \int_0^{L^2}\Bigg[-\sqrt{\frac{2}{\pi}}\int_{\sqrt{a}/\sigma}^{\infty}e^{-x^2/2}\dd x + \left(1 -  \frac{\sqrt{a}}{L}\right)\Bigg] \dd a \\
&= \sqrt{\frac{2}{\pi}}\int_0^{L/\sigma} (L^2-\sigma^2x^2)e^{-x^2/2} \dd x - \frac{2}{3}L^2.
\end{align*}
Note that for $t = t(L) = \frac{L^2}{\sigma^2} = \frac{6L^2}{(L+1)(2L+1)}$, the expression $\frac{g(L^2)}{\sigma^2}$ becomes
\[
h(t) = \sqrt{\frac{2}{\pi}}\int_0^{\sqrt{t}} (t-x^2)e^{-x^2/2} \dd x - \frac{2}{3}t.
\]
We have,
\[
h'(t) = \sqrt{\frac{2}{\pi}}\int_0^{\sqrt{t}} e^{-x^2/2} \dd x - \frac{2}{3}.
\]
For $L \geq 7$, we have $t \geq t_0 = t(7) = \frac{49}{20}$. We check that $h'(t_0) = h'(\frac{49}{20})> 0.2$ and since $h'$ is increasing, $h'(t)$ is positive for $t \geq t_0$, hence $h(t) \geq h(t_0) = h(\frac{49}{20}) > 0.01$ for $t \geq t_0$. Consequently, $g(L^2) > 0$ for every $L \geq 7$, which completes the proof for $L \geq 7$.

It remains to address the cases $2 \leq L \leq 6$. Here lower-bounding $f$ by $g$ incurs too much loss, so we show that $f$ is nonnegative on $[0,L^2]$ by direct computations. First note that $f'(a)$ (see \eqref{eq:f'}) is strictly increasing on each interval $a \in (b^2,(b+1)^2)$, $b \in \{0,1,\ldots, L-1\}$. Clearly $f'(0+) = 0$ and we check that $\theta_{L,b} = f'(b^2+) > 0$ for every $b \in \{1,\ldots,L-2\}$ and $3 \leq L \leq 6$ (see Table \ref{tab:f'}), so $f(a)$ is strictly increasing for $a \in (0,(L-1)^2)$. Since $f(0) = 0$, this shows that $f(a) > 0$ for $a \in (0,(L-1)^2)$. On the interval $((L-1)^2,L^2)$, we use the convexity of $f$ and we lower-bound $f$ by its tangent at $a = (L-1)^2+$ with the slope $\theta_{L,L-1}$ (which is negative), that is $f(a) \geq \theta_{L,L-1}(a - (L-1)^2) + f((L-1)^2)$. It remains to check that $v_L = \theta_{L,L-1}(2L-1) + f((L-1)^2)$, the values of the right hand side at the end point $a = L^2$, are positive. We have, $v_2 > 0.2$, $v_3 > 0.7$, $v_4 > 1.2$, $v_5 > 1.9$, $v_6 > 2.6$. This finishes the proof.
\end{proof}

\begin{table}[!ht]
\begin{center}
\caption{Lower bounds on the values of the slopes $\theta_{L,b} = f'(b^2+)$.}
\label{tab:f'}
\begin{tabular}{r|ccccc}
 & $b=1$ & $b=2$ & $b=3$ & $b=4$ \\\hline
$\theta_{3,b}$ & $0.02$\\
$\theta_{4,b}$ & $0.03$ & $0.03$\\
$\theta_{5,b}$ & $0.03$ & $0.05$ & $0.03$\\
$\theta_{6,b}$ & $0.03$ & $0.05$ & $0.05$ & $0.02$\\
\end{tabular}
\end{center}
\end{table}

\begin{remark}\label{rem:iid-2-p>3}
We can drop the assumption in Theorem \ref{thm:2-p>3} of the $X_i$ being identically distributed and only assume their independence (we stated it in the i.i.d. case for simplicity). The proof does not change: we only have to choose the independent Gaussian random variables $G_i$ to be such that $\E|G_i|^2 = \E|X_i|^2$ and then \eqref{eq:X^2vsG^2}, hence \eqref{eq:XvsG} holds for each $X_i$.
\end{remark}

\subsection{First and second moments: Proof of Theorem \ref{thm:L1-L2}}

For $a_1 = 1$, $a_2 = \dots = a_n = 0$, inequality $N_1(a) \geq c_1N_2(a)$ becomes equality, so the value of the constant $c_1$ is sharp. To prove the inequality, we shall closely follow Haagerup's approach from \cite{Haa}. Note that $Y$ has the same distribution as $\theta\varepsilon R$, where $\theta$ is a Bernoulli random variable with parameter $1-\rho_0$, $\varepsilon$ is a symmetric random sign, $R$ is a positive random variable and $\theta, \varepsilon$ and $R$ are independent (the law of $R$ is the same as the law of $|X|$ conditioned on $X \neq 0$). Let $\phi_Y(t) = \E e^{itY}$ be the characteristic function of $Y$. We have
\begin{align*}
\phi_Y(t) &= \rho_0 + (1-\rho_0)\E \cos(tR) \geq \rho_0 -(1-\rho_0) = 2\rho_0 -1 \geq 0.
\end{align*}
We also define
\[
F(s) = \frac{2}{\pi}\int_0^\infty\left[1 - \left|\phi_Y\left(\frac{t}{\sqrt{s}}\right)\right|^s\right]\frac{dt}{t^2}, \qquad s \geq 1.
\]
By symmetry, without loss of generality we can assume that $a_1, \ldots, a_n$ are positive with $\sum a_j^2 = 1$. By Lemma 1.2 from \cite{Haa} and independence,
\begin{align*}
N_1(a) = \E\left|\sum_j a_j Y_j\right| &= \frac{2}{\pi}\int_0^\infty \left[ 1 - \prod_j \phi_Y(a_jt) \right] \frac{dt}{t^2}.
\end{align*}
By the AM-GM inequality, 
$
\prod \phi_Y(a_jt) \leq \sum a_j^{2}|\phi_Y(a_jt)|^{a_j^{-2}},
$
thus
$
N_1(a) \geq \sum_j a_j^2F(a_j^{-2}).
$
If we show that
\begin{equation}\label{eq:F>F(1)}
F(s) \geq F(1), \qquad s \geq 1,
\end{equation}
then
\[
N_1(a) \geq \sum_j a_j^2F(1) = F(1) = \frac{F(1)}{\sqrt{\E |Y|^2}}N_2(a).
\]
Since $\phi_Y$ is nonnegative, using again Lemma 1.2 from \cite{Haa}, we have 
\[
F(1) = \frac{2}{\pi}\int_0^\infty\left[1 - \left|\phi_Y\left(t\right)\right|\right]\frac{dt}{t^2} = \frac{2}{\pi}\int_0^\infty\left[1 - \phi_Y\left(t\right)\right]\frac{dt}{t^2} = \E|Y|,
\]
so the proof of $N_1(a) \geq c_1N_2(a)$ with $c_1 = \|Y\|_1/\|Y\|_2$ is finished. 

It remains to show \eqref{eq:F>F(1)}. For a fixed $s \geq 1$, the left hand side
\[
F(s) = \frac{2}{\pi}\int_0^\infty\left[1 - \left|\rho_0 + (1-\rho_0)\E\cos\left(\frac{tR}{\sqrt{s}}\right)\right|^s\right]\frac{dt}{t^2}
\]
is concave as a function of $\rho_0$, whereas the right hand side $F(1) = \E|Y| = (1-\rho_0)\E R$ is linear as a function of $\rho_0$. Therefore, it is enough to check the cases: 1) $\rho_0 = 1$ which is clear, 2) $\rho_0 = 1/2$ which becomes
\[
\frac{2}{\pi}\int_0^\infty\left[1 - \left|\frac{1}{2} + \frac{1}{2}\E \cos\left(\frac{tR}{\sqrt{s}}\right)\right|^s\right]\frac{dt}{t^2} \geq \frac{1}{2}\E R.
\]
Using $\frac{\cos x +1}{2} = \cos^2(x/2)$ and then employing convexity, the left hand side can be rewritten and lower bounded as follows
\begin{align*}
\frac{2}{\pi}\int_0^\infty\left[1 - \left|\E \cos^2\left(\frac{tR}{2\sqrt{s}}\right)\right|^s\right]\frac{dt}{t^2} \geq \E\frac{2}{\pi}\int_0^\infty\left[1 - \left|\cos\left(\frac{tR}{2\sqrt{s}}\right)\right|^{2s}\right]\frac{dt}{t^2}.
\end{align*}
A change of variables $t = \sqrt{2}t'/R$ allows to write the right hand side as
\begin{align*}
\E\left[\frac{2}{\pi}\int_0^\infty\left[1 - \left|\cos\left(\frac{t'}{\sqrt{2s}}\right)\right|^{2s}\right]\frac{dt'}{t'^2}\frac{R}{\sqrt{2}}\right] = \frac{\E R}{\sqrt{2}}F_{\text{Haa}}(2s),
\end{align*}
where $F_{\text{Haa}}(s) = \frac{2}{\pi}\int_0^\infty\left[1 - \left|\cos\left(\frac{t}{\sqrt{s}}\right)\right|^{s}\right]\frac{dt}{t^2}$ is Haagerup's function (see Lemma 1.3 and 1.4 in \cite{Haa}). He showed therein that it is increasing, so for $s \geq 1$, we get $F_{\text{Haa}}(2s) \geq F_{\text{Haa}}(2) = \frac{1}{\sqrt{2}}$ and this finishes the proof.

\begin{remark}\label{rem:L1-L2-otherp}
Thanks to Remark 2.5 from \cite{Haa}, the same proof also works if we replace the first moment by $p_0$-th one, where $p_0 = 1.847...$ is the unique solution to $\Gamma(\frac{p+1}{2}) = \frac{\sqrt{\pi}}{2}$, $p \in (0,2)$. The cases of other values of $p \in (1,2)$ have been elusive. 
\end{remark}


\section{Necessity of the restrictions on $\rho_0$}\label{sec:rho0}

We use the notation from \eqref{eq:def-X} and \eqref{eq:def-N}. We derive some  necessary conditions on $\rho_0$, justifying to some extent our restrictions on $\rho_0$ made in Theorems \ref{thm:Schur}, \ref{thm:2-p>3} and \ref{thm:L1-L2}.

\begin{remark}\label{rem:rho-23}
For Theorem \ref{thm:Schur} to hold, we necessarily have $\frac{\dd}{\dd \lambda}N_3^3(\sqrt{\lambda},\sqrt{1-\lambda}) \geq 0$ for $\lambda \in (0,\frac{1}{2})$. Letting $\lambda \to 0+$ yields $(1-\rho_0)(1 - 2\rho_0) \geq 0$, hence $\rho_0 \leq \frac{1}{2}$.
\end{remark}

\begin{remark}
In Theorem \ref{thm:2-p>3}, a necessary condition on $\rho_0$ is $\rho_0 \leq 1 - \frac{27\pi}{128} = 0.33732..$. This follows from $N_3(1) \leq \|G\|_3N_2(1)$ with $L \to \infty$.
\end{remark}

\begin{remark}
In Theorem \ref{thm:L1-L2}, a necessary condition on $\rho_0$ is $\rho_0 \geq \sqrt{2}-1$. This follows from $N_1(1,1) \geq c_1N_2(1,1)$ applied to $Y = X$, which is equivalent to $(1-\rho_0)\frac{2L+1}{3L} \geq 2-\sqrt{2}$, so $L=1$ gives $\rho_0 \geq \sqrt{2}-1$.  
\end{remark}

Thus the restriction in Theorem \ref{thm:Schur} is sharp, while those in Theorems \ref{thm:2-p>3}, \ref{thm:L1-L2} are by-products of our proofs and can perhaps be improved. We believe the optimal ones are indicated above (for the following reasons: one can check that the case $n=2$ and $Y=X$ of Theorem \ref{thm:L1-L2} holds for $\rho_0 \in [\sqrt{2}-1,1)$; moreover, in the context of Theorem \ref{thm:2-p>3}, $N_4(1) \leq \|G\|_4N_2(1)$ with $L\to\infty$ is a sufficient condition for $N_q(a) \leq \frac{\|G\|_q}{\|G\|_p}N_p(a)$ to hold for all $2 \leq p < q$ even integers -- see Remark \ref{rem:even}).


\end{document}